\documentclass[reqno, 11pt]{amsart}
\usepackage{mathtools}
\usepackage{amsmath}
\usepackage{amssymb}
\usepackage{yhmath}
\usepackage{graphicx}
\usepackage{ mathrsfs }
\usepackage{bbm}
\usepackage{xcolor}
\usepackage{tikz-cd}
\usepackage{tikz}
\usetikzlibrary{patterns}
\usepackage{setspace} 

\usepackage[T1]{fontenc}
\usepackage[utf8]{inputenc}
\usepackage[french,english]{babel}

\makeatletter
\newenvironment{abstracts}{%
  \ifx\maketitle\relax
    \ClassWarning{\@classname}{Abstract should precede
      \protect\maketitle\space in AMS document classes; reported}%
  \fi
  \global\setbox\abstractbox=\vtop \bgroup
    \normalfont\Small
    \list{}{\labelwidth\z@
      \leftmargin3pc \rightmargin\leftmargin
      \listparindent\normalparindent \itemindent\z@
      \parsep\z@ \@plus\p@
      
      \itemsep\medskipamount
    }%
}{%
  \endlist\egroup
  \ifx\@setabstract\relax \@setabstracta \fi
}

\newcommand{\abstractin}[1]{%
  \otherlanguage{#1}%
  \item[\hskip\labelsep\scshape\abstractname.]%
}
\makeatother

\DeclareMathAlphabet{\mathpzc}{OT1}{pzc}{m}{it}

\usepackage{thmtools}
\usepackage{thm-restate}

\usepackage{caption}

\newtheorem{thm}{Theorem}[section]
\newtheorem*{prop*}{Proposition}
\newtheorem{cor}[thm]{Corollary}
\newtheorem*{cor*}{Corollary}
\newtheorem*{lem*}{Lemma}

\numberwithin{thm}{section}

\theoremstyle{definition}

\newtheorem{Def}[thm]{Definition}
\newtheorem{que}{Question}[section]

\numberwithin{equation}{section}

\def\scrF{{\mathcal F}}

\def\i{\mathrm{i}}

\def\G{\operatorname{G{}}}

\def\vol{\operatorname{vol}}

\def\Onder#1#2#3#4#5{#1 \setbox0=\hbox{$#1$}\setbox1=\hbox{$#2$}
       \dimen0=.5\wd0 \dimen1=\dimen0 \dimen2=\dp0 \dimen3=\dimen2
       \advance\dimen0 by .5\wd1 \advance\dimen0 by -#4
       \advance\dimen1 by -.5\wd1 \advance\dimen1 by -#4
       \advance\dimen2 by -#3 \advance\dimen2 by \ht1
       \advance\dimen2 by 0.3ex \advance\dimen3 by #5
        \kern-\dimen0\raisebox{-\dimen2}[0ex][\dimen3]{\box1}
       \kern\dimen1}

\newcommand{\F}{\mathcal{F}}

\renewcommand{\fg}{\mathfrak{g}}
\newcommand{\fk}{\mathfrak{k}}
\newcommand{\fp}{\mathfrak{p}}
\newcommand{\fa}{\mathfrak{a}}

\newcommand{\ba}{\backslash}

\newcommand{\op}{\operatorname}

\newcommand{\la}{\lambda}
\renewcommand{\epsilon}{\varepsilon}

\renewcommand{\G}{\Gamma}
\newcommand{\Ga}{\Gamma}
\newcommand{\cal}{\mathcal}

\newcommand{\be}{\begin{equation}}
\newcommand{\ee}{\end{equation}}
\newcommand{\ga}{\gamma}

\renewcommand{\c}{\mathbb C}
\renewcommand{\la}{\langle}
\newcommand{\ra}{\rangle}

\newcommand{\cM}{\mathcal M}
\begin{document}

\title[Bottom of the spectrum]{Infinite volume and atoms at the bottom of the spectrum}

\author{Sam Edwards, Mikolaj Fraczyk, Minju Lee*, Hee Oh}

\address{Department of Mathematical Sciences, Durham University, Lower Mountjoy, DH1 3LE Durham, United Kingdom}
\address{Mathematics department, University of Chicago, Chicago, IL 60637, USA}
\address{Mathematics department, University of Chicago, Chicago, IL 60637, USA}
\address{Mathematics department, Yale university, New Haven, CT 06520, USA}

\thanks{Oh was supported in part by NSF grant DMS-1900101}

\begin{abstracts}
\abstractin{english}
Let $G$ be a higher rank simple real algebraic group, or more generally, any semisimple real algebraic group with no rank one factors  and $X$ the associated Riemannian symmetric space.
For any Zariski dense discrete subgroup $\Ga<G$, 
we prove that $\op{Vol}(\Ga\ba X)=\infty$  if and only if  no positive Laplace
 eigenfunction  belongs to $L^2(\Ga\ba X)$, or equivalently, the bottom of the $L^2$-spectrum is not an atom of the spectral measure of the negative Laplacian. This contrasts with the rank one situation where the square-integrability of the base eigenfunction is determined by the size of the critical exponent relative to the volume entropy of $X$.
 \textbf{\\$\!$ \\ Volume infini et atomes au bas du spectre}
\abstractin{french}
Soit $G$ un groupe algébrique réel simple de rang supérieur, ou plus généralement un groupe algébrique réel semi-simple sans facteurs de rang un et $X$ l'espace symétrique riemannien associé. 
Pour tout sous-groupe discret dense de Zariski $\Gamma<G$, on prouve que $\mathrm{Vol}(\Gamma\backslash X)=\infty$ si et seulement si aucune fonction propre de Laplacien positive appartient à $L^2(\Gamma\backslash X)$, ou de manière équivalente, le bas du spectre $L^2$ n'est pas un atome de la mesure spectrale du Laplacien négatif.
Cela contraste avec la situation de rang un où l'intégrabilité au carré de la fonction propre de base est déterminée par la taille de l'exposant critique par rapport à l'entropie volumique de $X$.
\end{abstracts}
\maketitle





\keywords{}


\section{Introduction}
Let $\cM$ be a complete Riemannain manifold and let $\Delta$ denote the Laplace-Beltrami operator on $\cM$.
Define the real number $\lambda_0(\cM)\in [0, \infty)$ by
 \be\label{la33} \lambda_0(\cM) := \inf\left\lbrace \frac{\int_{\cM}\|\text{grad} \, f \|^2\,d\vol}{\int_{\cM}|f|^2\,d\vol}\,:\,f\in C^\infty_c(\cM)\right\rbrace,\ee
 where $C^\infty_c(\cM)$ denotes the space of all smooth functions with compact support.
This number $\lambda_0(\cM)$ is known as the bottom of the $L^2$-spectrum of the negative Laplacian $-\Delta$ and separates the $L^2$-spectrum and the positive spectrum \cite[p. 329]{Su} (Fig. 1). 
\begin{figure}[ht]  \includegraphics [height=2.5cm]{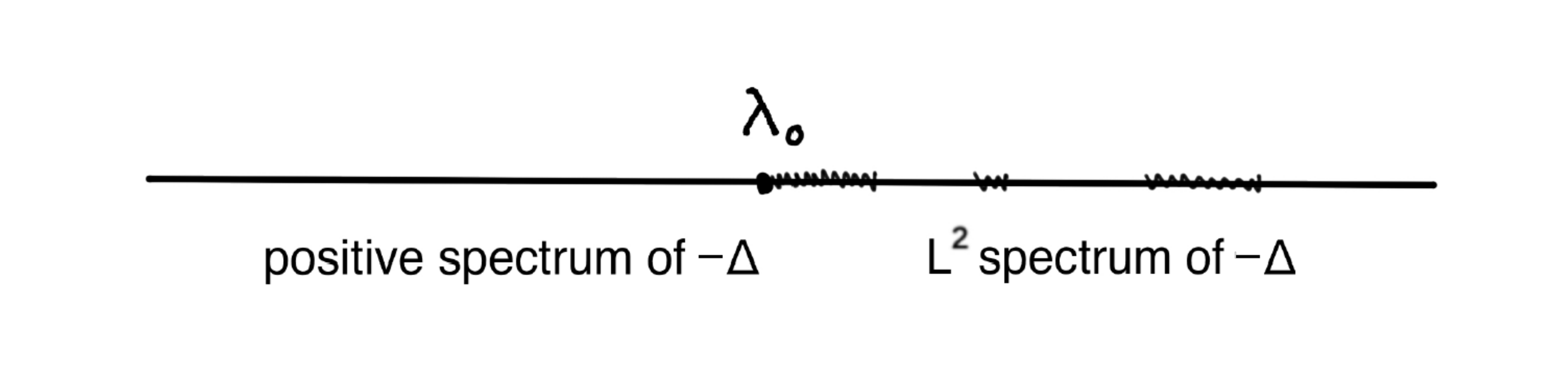}
\caption{$\lambda_0$ separates the $L^2$ and positive spectrum}\end{figure} 
More precisely, let $L^2(\cM)$
denote the space of all square-integrable functions  with respect to the inner product $\la f_1, f_2\ra=\int_{\cM} f_1 \overline{f_2} \,d\vol$.
 Let $W^1(\cM)\subset L^2(\cM)$ denote the closure of $C_c^{\infty}(\cM)$ with respect to the norm $$\|f\|_{W^1}=\left(\int_{\cM} f^2\,d\vol + \int_{\cM} \| \op{grad}f\|^2\,d\vol \right)^{1/2}.$$ There exists a unique self-adjoint operator on the space $W^1(\cM)$ extending the Laplacian $\Delta$ on $C^\infty_c(\cM)$,
 which we also denote by $\Delta$ (cf. \cite[Chapter 4.2]{Gr}).  The $L^2$-spectrum of $-\Delta$ is the set of all $\lambda\in \c$ such that $\Delta+\lambda $ does not have a bounded inverse
 $(\Delta+\lambda)^{-1}:L^2(\cM)\to W^1(\cM)$. 
Sullivan showed that the $L^2$-spectrum of $-\Delta$ contains
$\lambda_0(\cM)$ and is contained in the positive ray $[\lambda_0 (\cM),\infty),$
that is, $\lambda_0(\cM)$ is the bottom of the $L^2$-spectrum,
and moreover,  there are no positive eigenfunctions with eigenvalue strictly bigger than $\lambda_0(\cM)$ \cite[Theorem 2.1 and 2.2]{Su} (see Fig.\ 1).
We will call an eigenfunction with eigenvalue $\lambda_0(\cM) $ a {\it base eigenfunction}. Note that the absence of a base eigenfunction in $L^2(\cM)$ is the same as the absence of a positive eigenfunction in $L^2(\cM)$
\cite[Cor. 2.9]{Su}.

In this paper, we are concerned with locally symmetric spaces. Let $G$ be a connected semisimple real algebraic group and $(X, d)$ the associated Riemannian symmetric space. Let $\Ga<G$ be a discrete torsion-free subgroup and let $\cM=\Ga\ba X$ the corresponding locally symmetric manifold.

For a rank one locally symmetric manifold $\cM=\Ga\ba X$, the relation between $\lambda_0(\cM)$
and the critical exponent\footnote{the abscissa of convergence of the Poincare series $s\mapsto \sum_{\gamma\in \Gamma}e^{-s d(o, \gamma o)}$, $o\in X$.}  $\delta_\Ga$ is well-known:
if we denote by $D=D_X$ the volume entropy of $X$,
then \begin{equation*} \lambda_0(\cM) =\begin{cases} D^2/4 & \text{ if $\delta_\Ga\le D/2$}\\
\delta_\Ga (D-\delta_\Ga)
&\text{otherwise}\end{cases}\end{equation*}
(\cite{E1}-\cite{E3}, \cite{P1}-\cite{P3}, \cite{Su}, \cite{Cor}). We refer to
 (\cite{Leu}, \cite{AZ}, \cite{Co}, \cite{EO}, \cite{WW2}) for extensions
 of these results to higher ranks. We remark that when $G$ has Kazhdan’s property (T)  (cf. \cite[Theorem
7.4.2]{Zi}), we have
 $\op{Vol}(\cM)=\infty$
 if and only if $\lambda_0(\cM)>0$ (\cite{Cor}, \cite{Leu}).

\medskip 
The goal of this article is to study the square-integrability of a base eigenfunction of locally symmetric manifolds. The space of square-integrable base eigenfunctions is at most one dimensional and generated by a {\it positive} function when non-trivial \cite{Su}. Based on this positivity property and using their theory of conformal measures on the geometric boundary, Patterson and Sullivan showed that if $\cM$ is a geometrically finite real hyperbolic $(n+1)$-manifold,
then
$\cM$ has a square-integrable base eigenfunction if and only if the critical exponent $\delta_\Ga$ 
is strictly greater than $n/2$ (\cite{Pa}, \cite{Su2}, \cite[Theorem 2.21]{Su}). More generally, the formula for $\lambda_0(\cM)$ given above, together with \cite[Corollary 3.2]{Ha} (cf.\ also \cite{Li}) and \cite[Theorem 1.1]{WW}, implies  that any rank one geometrically finite manifold $\cM$  has a square-integrable base eigenfunction if and only if the critical exponent $\delta_\Ga$ is strictly greater than $ D_X/2$.

The main theorem of this paper is the following surprising higher rank phenomenon that contrasts with the rank one situation:

\begin{thm}\label{main}
    Let $G$ be a connected semisimple real algebraic group with no rank one factors.
    For any Zariski dense discrete torsion-free subgroup $\Ga<G$, we have
    $\op{Vol}(\Ga\ba X)=\infty$  if and only if
    $\Ga\ba X$ does not possess any square-integrable positive Laplace eigenfunction, that is, 
   $\lambda_0(\Ga\ba X)>0$ is not an atom for the spectral measure of $-\Delta$. 
\end{thm}

In other words, when $\text{Vol}(\Ga\ba X)=\infty,$ no base eigenfunction is square-integrable  (see also Theorem \ref{main2} for a more general version). A special case of this theorem for Anosov subgroups of higher rank semisimple Lie groups was proved in \cite[Theorem 1.8]{EO}.
See Theorem \ref{main2} for a more general version.

Our proof of Theorem \ref{main} is based on the higher rank version of Patterson-Sullivan theory introduced by Quint \cite{Quint}, with a main new input being the recent theorem of  Fraczyk and Lee (Theorem 
\ref{fl}, \cite{FL}). Suppose that
$\text{Vol}(\Ga\ba X)=\infty$ and  a base eigenfunction
is square-integrable. Using Sullivan's work \cite{Su}, it was then shown by Edwards and Oh \cite{EO} 
that there exists a $\Ga$-conformal density $\{\nu_x: x\in X\}$
on the Furstenberg boundary of $G$ (see Definition \ref{conf_def}) such that
any such base eigenfunction is proportional to the function $E_\nu$
given by
\be\label{sm} E_\nu(x)=|\nu_x|\quad\text{for all $x\in X$}.\ee 
Moreover, the following higher rank version of the smearing theorem of Thurston and Sullivan (\cite{Su2}, \cite{Su3}) was also obtained by Edwards-Oh \cite{EO} (see Theorem \ref{smear}):
$$|{\mathsf m}_{\nu, \nu}|\ll \int_{\Ga\ba X} |E_\nu|^2 dx, $$
where ${\mathsf m}_{\nu, \nu}$ is a generalized Bowen-Margulis-Sullivan measure
on $\Ga\ba G$ corresponding to the pair $(\nu, \nu)$; see Definition \ref{measure}.
On the other hand,
the recent theorem of Fraczyk and Lee (Theorem \ref{fl}, \cite{FL}) which describes all discrete subgroups admitting finite BMS measures implies that $|{\mathsf m}_{\nu, \nu}|=\infty$, and consequently, $E_\nu \notin L^2(\Ga\ba X)$, yielding a contradiction. We remark that the integrand on the right hand side of \eqref{sm} can be replaced by an $O(1)$-neighborhood of the support of $\mathsf m_{\nu, \nu}$ and
Sullivan used the rank one version of this to deduce the finiteness of the BMS measure 
$\mathsf m_{\nu, \nu}$ attached to the (unique) Patterson-Sullivan measure $\nu$ from the the growth control of the base eigenfunction for $\Ga$ geometrically finite \cite{Su2}.

\medskip 

 We close the introduction by presenting two related questions on the $L^2$-spectrum.
When $\Ga<G$ is geometrically finite in a rank one Lie group and
 there is no positive square-integrable eigenfunction, there are no Laplace eigenfunctions in $L^2(\Ga\ba X)$ and the quasi-regular representation $L^2(\Ga\ba G)$ is tempered\footnote{This means that
 $L^2(\Ga\ba G)$ is weakly contained in $L^2(G)$, or equivalently, every matrix coefficient of $L^2(\Ga\ba G)$  is $L^{2+\epsilon}(G)$-integrable for any $\epsilon>0$.} 
 (\cite{P1}, \cite{Su2}, \cite{CI}, \cite{LP}). In view of this, we ask the following question: let $G$ be a semisimple real algebraic group with no rank one factors and
 $\Ga<G$ be a Zariski dense discrete subgroup.
\begin{que}\label{Q}
    When $\Gamma<G$ is not a lattice, 
can there exist any Laplace eigenfunction in $L^2(\Ga\ba X)$?
    \end{que}

\medskip 

\noindent{\bf Acknowledgements}
We would like to thank Peter Sarnak and David Fisher for their interests and useful comments. We thank the anonymous referee for helpful remarks.

\section{Positive eigenfuntions and conformal measures}
Let $G$ be a connected semisimple real algebraic group.
We fix, once and for all, a Cartan involution $\theta$ of the Lie algebra $\mathfrak{g}$ of $G$, and decompose $\fg$ as $\mathfrak g=\mathfrak k\oplus\mathfrak{p}$, where $\fk$ and $\fp$ are the $+ 1$ and $-1$ eigenspaces of $\theta$, respectively. We denote by $K$ the maximal compact subgroup of $G$ with Lie algebra $\fk$. We also choose a maximal abelian subalgebra $\fa$ of $\mathfrak p$.
We denote by $\la\cdot,\cdot\ra$ and $\| \cdot \|$ respectively
 the Weyl-group invariant inner product and norm on $\mathfrak a$ induced from the Killing form on $\frak g$.  We denote by $X=G/K$ the corresponding Riemannian symmetric space equipped
with the Riemannian metric $d$ induced by the Killing form on $\fg$.
The Riemannian volume form on $X$ is denoted by $d\op{vol}$ . We also use $dx$ to denote this volume form, as well as for the Haar measure on $G$.

Let $A:=\exp \mathfrak a$.
Choosing a closed positive Weyl chamber $\fa^+$ of $\fa$,  let $A^+=\exp \mathfrak a^+$. The centralizer of $A$ in $K$ is denoted by $M$, and we set 
$N$ to be the maximal  horospherical subgroup for $A$ so that
 $\log (N)$ is the sum of all positive root subspaces for our choice of $\fa^+$. 
We set $P=MAN$, which is a minimal parabolic subgroup of $G$. The quotient $$\F=G/P$$ is known as the Furstenberg boundary of $G$, and since $K$ acts transitively on $\F$ and $K\cap P=M$, we may identify $\cal F$ with $K/M$.

Let $\Sigma^+$ denote the set of all  positive roots for $(\fg, \fa^+)$. We also write $\Pi\subset \Sigma^+$ for the set of all simple roots.
For any $g\in G$, there exists a unique element $\mu(g)\in \fa^+$ such that $g\in K \exp \mu(g) K$. The map $\mu:G\to \fa^+$ is called the Cartan projection. Setting $o=[K]\in X$, we then have
$\|\mu(g)\|=d(go, o)$ for all $g\in G$. Throughout the paper we will identify functions on $X$ with
right $K$-invariant functions on $G$. For each $g\in G$, we define the following {\it visual} maps:
   \be\label{visual} g^+:=gP\in \cal F\quad\text{and}\quad g^-:=gw_0P\in \cal F,\ee 
where $w_0$ denotes the longest Weyl group element, i.e. the Weyl group element such that $\op{Ad}_{w_0}\mathfrak a^+= -\mathfrak a^+$.
 The unique open $G$-orbit $\F^{(2)}$ in $\F\times \F$ under the diagonal $G$-action is given by
$\F^{(2)}=G(e^+, e^-)=\{(g^+, g^-)\in \F\times \F: g\in G\}.$ Let $G=KAN$ be the Iwasawa decomposition, and define the Iwasawa cocycle $H:G\rightarrow \fa$ by the relation:
$$ g\in K\exp\big(H(g)\big)N.$$

The $\fa$-valued Busemann map is defined using the Iwasawa cocycle as follows: for all $g\in G$ and $[k]\in\scrF$ with $k\in K$, define
$$ \beta_{[k] }(g(o),h(o)):=H(g^{-1}k)-H(h^{-1}k)\in \fa \quad\text{
for all $g,h\in G$.}$$

\noindent{\bf Conformal measures.}
We denote by $\fa^*$ the space of all real-valued linear forms on $\fa$.
In the rest of this section, let $\Ga<G$ be a discrete subgroup.
The following notion of conformal densities was introduced by Quint \cite[Section 1.2]{Quint}, generalizing Patterson-Sullivan densities for rank one groups (\cite[Section 3]{Pa}, \cite[Section 1]{Su1}).
\begin{Def} \label{conf_def} Let $\psi\in\fa^*$.
\begin{enumerate}
    \item A  finite Borel
measure $\nu$ on $\scrF=K/M$ is said to be a $(\Ga,\psi)$-conformal measure (for the basepoint $o$) if for all $\ga\in \Ga$ and $\xi=[k]\in K/M$,
$$ \frac{d\ga_{\ast}\nu}{d\nu}(\xi)=e^{-\psi(\beta_\xi (\ga o,  o))},$$
where $\gamma_*\nu(Q)=\nu(\gamma^{-1}Q)$ for
any Borel subset $Q\subset \cal F$.
\item A collection $\{\nu_x: x\in X\} $ of finite Borel measures on $\F$
is called a $(\Ga, \psi)$-conformal density if, 
for all $x, y\in X$, $\xi\in \F$ and $\ga\in \Ga$,
\be \label{c0} \frac{d\nu_x}{d\nu_y}(\xi)= e^{-\psi (\beta_\xi (x,y))}\quad\text{and}\quad d\gamma_* \nu_x =d\nu_{\ga (x)}.\ee 
\end{enumerate}

\end{Def}

A $(\Gamma,\psi)$-conformal  measure $\nu$ defines a $(\Ga, \psi)$-conformal density $\{\nu_x: x\in X\}$ by the formula:
$$d\nu_x(\xi) =e^{-\psi(\beta_\xi (x, o))} d\nu(\xi) ,$$
and conversely any $(\Ga,\psi)$-conformal density $\{\nu_x\}$
is uniquely determined by its member $\nu_o$ by \eqref{c0}. By a $\Ga$-conformal measure on $\F$, we mean a $(\Ga, \psi)$-conformal measure for some $\psi\in \fa^*$.

\begin{Def} Let $\psi\in \fa^*$. Associated to a $(\Ga, \psi)$-conformal measure $\nu$ on $\cal F$, we define the following function $E_\nu$ on $G$: for $g\in G$,\be\label{ef} E_{\nu}(g):=|\nu_{g(o)}|=\int_{\scrF} e^{-\psi\big(H(g^{-1}k)\big)}\,d\nu([k]).\ee
We remark that this is same as the transformation introduced in \cite{Hel}.
Since $|\nu_{\ga(x)}|=|\nu_x|$ for all $\ga \in \Ga$ and $x\in X$,
 the left $\Ga$-invariance and right $K$-invariance
 of $E_\nu$ are clear. Hence we may consider $E_\nu$
 as a $K$-invariant function on $\Gamma\ba G$, or, equivalently, as a
 function on $\Gamma\ba X$.
\end{Def}
Let $\cal D=\cal D(X)$ denote the ring of all $G$-invariant differential operators on $X$.
For each $(\Ga, \psi)$-conformal measure $\nu$, $E_\nu$ is a joint eigenfunction of $\cal D$ and conversely,
any {\it positive} joint eigenfunction on $\Ga\ba X$ arises as $E_\nu$
for some $(\Ga, \psi)$-conformal measure $\nu$ \cite[Proposition 3.3]{EO}.

Let $\Delta$ denote the Laplace-Beltrami operator on $X$ or on $\Gamma\ba X$.
Since $\Delta$ is an elliptic differential operator, an eigenfunction is always smooth. 
We say a smooth function $f$ is $\lambda$-harmonic if $$-\Delta f=\lambda f.$$

 Define the real number $\lambda_0=\lambda_0(\Gamma\ba X)\in [0,\infty)$ as follows: \be\label{ll} \lambda_0:= \inf\left\lbrace \frac{\int_{\Gamma\ba X}\|\text{grad} \, f \|^2\,d\vol}{\int_{\Gamma\ba X}|f|^2\,d\vol}\,:\,f\in C^\infty_c(\Gamma\ba X),\; f\neq 0 \right\rbrace .\ee 
 
We call a $\lambda_0$-harmonic function on $\Ga\ba X$ a base eigenfunction. In general, a $\lambda$-harmonic function need not be a joint eigenfunction for the ring $\cal D(X)$.
However, a square-integrable $\lambda_0$-harmonic function turns out to be a {\it positive} joint eigenfunction, up to a constant multiple.
The following is obtained in \cite[Corollary 6.6, Theorem 6.5]{EO} using Sullivan's work \cite{Su} and \cite{La}.
\begin{thm}\cite{EO}\label{bbb}
If a base eigenfunction $\phi_0$ belongs to $L^2(\Ga\ba X)$, then there exists
$\psi\in \fa^*$ and a $(\Ga, \psi)$-conformal measure $\nu$ on $\F$ such that
$\phi_0$
is proportional to $E_\nu$.
\end{thm}

Here the space $L^2(\Ga\ba X)$
consists of square-integrable functions with respect to
the inner product $\la f_1, f_2\ra=\int_{\Ga\ba X} f_1 \overline{f_2} \,d\vol$.

\section{Higher rank smearing theorem}
Let $G$ be a connected semisimple real algebraic group and $\Ga<G$ be a discrete subgroup.
We recall the definition of a generalized Bowen-Margulis-Sullivan measure, as was defined in \cite[Section 3]{ELO}.

Fix   a pair of linear forms $\psi_1, \psi_2\in \mathfrak a^*$. 
 Let $\nu_1$ and $\nu_2$ be respectively
 $(\Gamma,\psi_1)$ and $(\Gamma,\psi_2)$ conformal measures on $\F$.
Using the homeomorphism (called the Hopf parametrization)  $G/M\to \F^{(2)} \times \mathfrak a$
given by $gM \mapsto (g^+, g^-, b=\beta_{g^-}(o,go)) $, define the following locally finite Borel measure $\tilde {\mathsf m}_{\nu_1, \nu_2}$ on $G/M$ 
  as follows: for $g=(g^+, g^-, b)\in \F^{(2)}\times \mathfrak a$,
\begin{equation}\label{eq.BMS0}
d\tilde {\mathsf m}_{\nu_1, \nu_2} (g)=e^{\psi_1 (\beta_{g^+}(o, go))+\psi_2( \beta_{g^-} (o, go )) } \;  d\nu_{1} (g^+) d\nu_{2}(g^-) db,
\end{equation}
  where $db=d\ell (b) $ is the Lebesgue measure on $\mathfrak a$ induced from the inner product $\langle \cdot,\cdot\rangle$.
The measure $\tilde {\mathsf m}_{\nu_1, \nu_2}$ is left $\Gamma$-invariant and right $A$-semi-invariant:
for all $a\in A$,
\be\label{semi}
a_*\tilde {\mathsf m}_{\nu_1, \nu_2}=e^{(-\psi_1+\psi_2\circ \i)(\log a)}\,\tilde {\mathsf m}_{\nu_1, \nu_2}, \ee 
where $\i$ denotes the opposition involution\footnote{It is defined by 
$\i (u)= -\op{Ad}_{w_0} (u),$
where $w_0$ is the longest Weyl element.}  $\i:\mathfrak a \to \mathfrak a$ (cf.\ \cite[Lemma 3.6]{ELO}). The measure $\tilde {\mathsf m}_{\nu_1, \nu_2}$ 
gives rise to a left $\Ga$-invariant and right $M$-invariant measure on $G$ by integrating along the fibers of $G\to G/M$ with respect to the Haar measure on $M$.
By abuse of notation, we will also denote this measure by $\tilde {\mathsf m}_{\nu_1, \nu_2}$.
We denote by \be\label{measure} {\mathsf m}_{\nu_1, \nu_2}\ee 
 the measure on $\Gamma\ba G$ induced by $\tilde {\mathsf m}_{\nu_1, \nu_2}$, and call it the generalized BMS-measure associated to the pair $(\nu_1, \nu_2)$.

The following theorem was proved in \cite[Theorem 7.4]{EO}, extending
the smearing argument due to Sullivan and Thurston (\cite[Proposition 5]{Su2}, \cite[Proof of Theorem 4.1]{CI}) to the higher rank setting.

\begin{thm}[Edwards-Oh, \cite{EO}]
\label{smear} Let $\psi_1, \psi_2\in \fa^*$.
There exists a constant $c=c(\psi_1, \psi_2)>0$ such that
for any pair $(\nu_1, {\nu}_2)$ of $(\Ga,\psi_1)$ and $(\Ga,
\psi_2)$-conformal measures on $\F$ respectively,
$$|{\mathsf m}_{\nu_1,{\nu}_2}|\le 
c \int_{\text{$1$-neighborhood of $\op{supp}\mathsf m_{\nu_1, \nu_2}$}} E_{\nu_1}(x) E_{{\nu}_2}(x) \, dx  .$$
\end{thm}
Although \cite[Theorem 7.4]{EO} was stated so that $c$ depends on $\nu_1, \nu_2$,
the formula for $c$ given in its proof shows that $c$  depends only on the associated linear forms $\psi_1, \psi_2$. 
 
 An immediate corollary is as follows:
\begin{cor}
Let $\nu$ be a $\Ga$-conformal measure on $\F$.
If $|{\mathsf m}_{\nu, \nu}|=\infty$, then 
$$E_{\nu}\notin L^2(\Ga\ba X).$$
\end{cor}
\section{Proof of Main theorem}
As in Theorem \ref{main}, let 
 $G$ be a connected semisimple real algebraic group with no rank one factors
    and $\G<G$ be a Zariski dense discrete torsion-free subgroup.
We recall the following recent theorem:
\begin{thm}[Fraczyk-Lee, \cite{FL}] \label{fl} Suppose that $\op{Vol}(\Ga\ba X)=\infty$. Then for any pair $(\nu_1, \nu_2)$ of
$(\Ga, \psi)$ and $(\Ga, \psi\circ \i)$-conformal measures for some $\psi\in \fa^*$,
$${\mathsf m}_{\nu_1, \nu_2}(\Ga\ba G)=\infty.$$
\end{thm}

\begin{cor}\label{fl2}
If $\op{Vol}(\Ga\ba X)=\infty$, then for any pair $(\nu_1, \nu_2)$ of
$\Ga$-conformal measures,
${\mathsf m}_{\nu_1, \nu_2}(\Ga\ba G)=\infty$.
\end{cor}
\begin{proof}
For $k=1,2$, let $\nu_k$ be a $(\Ga, \psi_k)$-conformal measure with $\psi_k\in \fa^*$. Suppose $|{\mathsf m}_{\nu_1, \nu_2}|<\infty$.
Since $a_* {\mathsf m}_{\nu_1, \nu_2}=e^{\psi_1 (\log a) -\psi_2(\i \log a)} {\mathsf m}_{\nu_1, \nu_2}$ for all $a\in A$ by \eqref{semi},
it follows that 
$$|{\mathsf m}_{\nu_1, \nu_2}| = e^{\psi_1 (\log a) -\psi_2(\i \log a)} |{\mathsf m}_{\nu_1, \nu_2}|.$$
Since  $|{\mathsf m}_{\nu_1, \nu_2}|<\infty$, we must have $$\psi_2 =\psi_1\circ \i .$$ Therefore the claim follows from Theorem \ref{fl}.
    \end{proof}

\noindent{\bf Proof of Theorem \ref{main}}
Suppose that $\op{Vol}(\Ga\ba X)=\infty$ and $\phi_0$ is
a base eigenfunction in $L^2(\Ga\ba X)$.
By Proposition \ref{bbb}, we may assume that $\phi_0=E_\nu$
 for some $\Ga$-conformal measure $\nu$ on $\F$.
Now by Theorem \ref{smear} and Corollary \ref{fl2},
$$\infty=|{\mathsf m}_{\nu, \nu}|\ll \|E_{\nu}\|_2^2 .$$
This is a contradiction.     

   Indeed, using a more precise version of the main theorem of \cite{FL} in replacement of Theorem \ref{fl}, we obtain the following without the hypothesis on no rank one factors.
\begin{thm}\label{main2}
    Let $G$ be a connected semisimple real algebraic group and $\Ga<G$ be a Zariski dense
    discrete subgroup.
    If $\Gamma\ba X$ admits a square-integrable base eigenfunction, then $G=G_1G_2$,
    $\Ga$ is commensurable with $\Gamma_1 \Gamma_2 $
    where $G_1$ $($resp. $G_2)$ is an almost direct product of rank one $($resp. higher rank$)$ factors of $G$, $\Ga_1<G_1$ is a discrete subgroup and $\Ga_2<G_2$ is a lattice.
\end{thm}

\end{document}